\documentclass{amsart}

\usepackage{amsfonts,amssymb,amsmath,amsthm}
\usepackage{url, amsmath}
\usepackage{enumerate}
\usepackage[utf8]{inputenc}
\usepackage[T1]{fontenc}
\usepackage{graphicx}
\usepackage{epsf, subfigure, verbatim}
\usepackage{epsfig}
\usepackage{enumerate}

\newcommand{\closure}[2][3]{%
  {}\mkern#1mu\overline{\mkern-#1mu#2}}

\newtheorem{theorem}{Theorem}[subsection]
\newtheorem{lemma}[theorem]{Lemma}

\theoremstyle{definition}

\newtheorem{remark}[theorem]{Remark}
\numberwithin{equation}{subsection}

\author{Guillaume Poliquin}
\address{Département de mathématiques et de statistique \\
Université de Montréal, CP 6128 succ. Centre-Ville, Montréal\\
H3C 3J7, Canada.}
\email{gpoliquin@dms.umontreal.ca}
\thanks{Research supported by a NSERC scholarship}

\begin{document}
\noindent \let\thefootnote\relax

\title[Bounds for the $p$-Laplacian]{Principal frequency of the $p$-Laplacian and the inradius of Euclidean domains}

\begin{abstract}

We study the lower bounds for the principal frequency of the $p$-Laplacian on $N$-dimensional Euclidean domains. For $p>N$, we obtain a lower bound for the first eigenvalue of the $p$-Laplacian in terms of its inradius, without any assumptions on the topology of the domain. Moreover, we show that a similar lower bound can be obtained if $p > N-1$ assuming the boundary is connected. This result can be viewed as a generalization of the classical bounds for the first eigenvalue of the Laplace operator on simply connected planar domains.

\end{abstract}
\maketitle

\section{Introduction and main results}

\subsection{Physical models involving the $p$-Laplacian}

Let $\Omega$ be an $N$-dimensional Euclidean bounded domain. The $p$-Laplacian, where $1<p< \infty, \  p\neq 2$, is a nonlinear operator defined as $$\Delta_p f =  \operatorname{div}(|\nabla f|^{p-2} \nabla f),$$ for suitable $f$. Notice that the case $p=2$ corresponds to the well known Laplace operator. The $p$-Laplacian is used to model different physical phenomena, see for instance \cite{DTh, GR, Poli}.

The Laplacian can be used to describe the vibration of a homogeneous elastic membrane, such as a vibrating drum. The $p$-Laplace operator can also be used to model vibrating membrane, but composed of a nonelastic membrane under the load $f$,
\begin{eqnarray}
-\Delta_p (u) = f \qquad\mbox{ in } \Omega, \\
u=0 \quad\qquad \mbox{ on } \partial\Omega. \nonumber
\end{eqnarray}
The solution $u$ stands for the deformation of the membrane from the equilibrium position (see \cite{CEP, Si}). In that case, its deviation energy is given by $\int_\Omega |\nabla u|^p dx$. Therefore, a minimizer of the Rayleigh quotient,
$$ \frac{\int_\Omega |\nabla u|^p dx}{\int_\Omega |u|^p dx},$$
on $W_0^{1,p}(\Omega)$ satisfies $-\Delta_p (u) =  \lambda_{1,p}|u|^{p-2}u  \mbox{ in } \Omega.$ Here $\lambda_{1,p}$ is usually referred as the principal frequency of the vibrating nonelastic membrane.

\subsection{The eigenvalue problem for the $p$-Laplacian}

For $1 < p < \infty$, we study the following eigenvalue problem:
\begin{equation}\label{pLaplacian}
  \Delta_p u + \lambda |u|^{p-2}u=0 \mbox{ in } \Omega,
\end{equation}
where we impose the Dirichlet boundary condition and consider $\lambda$ to be the real spectral parameter. We say that $\lambda$ is an eigenvalue of $-\Delta_p$ if \eqref{pLaplacian} has a nontrivial weak solution $u_\lambda \in W^{1,p}_0(\Omega)$. That is, for any $v \in  C^\infty_0(\Omega)$,
\begin{equation}
\int_\Omega |\nabla u_\lambda|^{p-2} \nabla u_\lambda \cdot \nabla v - \lambda \int_\Omega |u_\lambda|^{p-2} u_\lambda v=0.
\end{equation}
The function $u_\lambda$ is then called an eigenfunction of $-\Delta_p$ associated to the eigenvalue $\lambda$. The case $N=1$ is better understood since explicit solutions in terms of beta functions are known (see \cite{D,E}).

If $N\geq2$, it is known that the first eigenvalue of the Dirichlet eigenvalue problem of the $p$-Laplace operator, denoted by $\lambda_{1,p}$, is characterized as,
\begin{equation}\label{var1}
\lambda_{1,p} = \min_{0 \neq u\in C^\infty_0(\Omega)} \left \{ \frac{ \int_\Omega |\nabla u|^p dx}{\int_\Omega |u|^p dx} \right \}.
\end{equation}
The infimum is attained for a function $u_{1,p} \in W^{1,p}_0(\Omega)$. In addition, $\lambda_{1,p}$ is simple and isolated. Moreover, the eigenfunction $u_{1}$  associated to $\lambda_{1,p}$ does not change sign, and it is the only such eigenfunction (a proof can be found in \cite{Lind1}).

Via for instance Lyusternick-Schnirelmann maximum principle, it is possible to construct $\lambda_{k,p}$ for $k\geq 2$ and hence obtain an increasing sequence of eigenvalues of \eqref{pLaplacian} tending to $\infty$. There exist other variational characterizations of these eigenvalues. However, no matter which variational characterization one chooses, it always remains to show that all the eigenvalues obtained that way exhaust the whole spectrum of $\Delta_p$.

\subsection{The principal frequency and the inradius}

Given an Euclidean domain $\Omega$, the inradius is defined as the radius of the largest ball inscribed in $\Omega$, denoted by $B_{\rho_{\Omega}}$ where $\rho_\Omega:=\sup\{r:\exists B_r \subset \Omega\}$. Obtaining an upper bound for the first eigenvalue of the $p$-Laplacian involving the inradius is immediate. Indeed, noticing that  $B_{\rho_{\Omega}} \subset \Omega$ and using the domain monotonicity property, we have that
$$
\lambda_{1,p}(\Omega)\leq \lambda_{1,p}(B_r) = \lambda_{1,p}(B_1)\rho_\Omega^{-p}.
$$

On the other hand, lower bounds involving the principal frequency are a greater challenge. A classical lower bound for the case the Laplacian is the Faber-Krahn inequality. It that can be adapted to the $p$-Laplacian, $\lambda_{1,p}(\Omega) \geq \lambda_{1,p}(\Omega^*),$ where $\Omega^*$ stands for the $n$-dimensional ball of same volume than $\Omega$ (see \cite[p. 224]{Lind2}).

It is also possible to obtain a lower bound for the first eigenvalue of the Laplace operator involving the inradius of the domain. That is
\begin{equation}\label{goal_lap} \lambda_{1,2}(\Omega) \geq \alpha_{N,2} \ \rho_{\Omega}^{-2},\end{equation}
where $\alpha_{N,2} >0$ is a positive constant. The latter is equivalent, in the theory of vibrating membranes, to knowing whether for a drum to produce an arbitrarily low note, it is necessary that one can inscribe an arbitrarily large circular drum.

In the case of simply connected planar domains, it is known that \eqref{goal_lap} holds with the constant $\alpha_{2,2}=\frac{\pi^2}{4}$ (see \cite{He, Mak, Hay, Oss}). More recently, the better constant $ \alpha_{2,2} \approx 0.6197$ was found using probabilistic methods in \cite{BaC}. A similar lower bound for domains of connectivity $k \geq 2$ also holds (see \cite{Cr}).

In higher dimensions, it was first noted by W.K. Hayman in \cite{Hay} that if $A$ is a ball with many narrow inward pointing spikes removed from it, then $\lambda_{1,2}(A) = \lambda_{1,2}(\mbox{Ball}),$ but in that case the inradius of $A$  tends to $0$. Therefore,  bounds of the type,
$$\lambda_{1,2}(\Omega) \geq \alpha_{N,2} \ \rho_{\Omega}^{-2},$$
are generally not possible to obtain even if $\Omega$ is assumed to be simply connected.

The aim of this paper is to study the following generalization to the case of the $p$-Laplace operator of the following inequality,
\begin{equation}\label{goal} \lambda_{1,p}(\Omega) \geq \alpha_{N,p} \ \rho_{\Omega}^{-p},\end{equation}
where $\alpha_{N,p} >0$ is a positive constant. The main results, stated in the Section \ref{MainResults}, were announced in \cite{Poli}. Other related lower bounds for the first eigenvalue were also discussed in \cite{Poli}.

\subsection{Main results}\label{MainResults}
A striking difference between the usual Laplace operator and the $p$-Laplacian is that it is possible to obtain bounds of the type \eqref{goal} in higher dimensions, as long as $p$ is "large enough" compared to the dimension. Indeed, Hayman's observation remains valid in the case $p \leq N-1$ since for such $p$, every curve has a trivial $p$-capacity. Recall that $p$-capacity can be defined for a compact set $K\subset B_r$, where $N > 2$, as
$$\operatorname{Cap}_p(K,B_r):= \inf\left\{ \int_{B_r} |\nabla \phi|^p dx, \ \ \phi \in C^\infty_0(B_r), \ \phi\geq 1 \ \mbox{ on } K \right\}.$$
On the other hand, for $p>N$, even a single point has a positive $p$-capacity (see \cite[Chapter 13, Proposition 5 and its corollary]{Maz}). Consequently, Hayman's counterexample no longer holds since removing a single point has an impact on the eigenvalues as it can be seen from \eqref{var1}. This leads to the following theorem:

\begin{theorem}\label{ThmPrincipal}
Let $p > N$ and let $\Omega$ be a bounded Euclidean domain. Then, there exists a positive constant $C_{N,p}$ such that
\begin{equation}\label{EqPrincipal}
\lambda_{1,p}(\Omega) \geq C_{N,p}\rho_{\Omega}^{-p}.
\end{equation}
\end{theorem}

It is known that if $p > N-1$, every curve has a positive $p$-capacity. Therefore, Hayman's counterexample does not work in that case as well. Nevertheless, points do not have a positive $p$-capacity if $N-1 < p \leq N$. Taking into account the latter observations, we get the following result:



\begin{theorem} \label{CurPosCap}
Let $\Omega$ be a bounded Euclidean domain such that $\partial \Omega$ is connected. If $ p > N-1$, then there exists a positive constant $C_{N,p}$ 
such that
\begin{equation} \lambda_{1,p}(\Omega) \geq C_{N,p}\rho_\Omega^{-p}.\end{equation}
\end{theorem}

Theorem \ref{CurPosCap} can be viewed as a generalization of classical results for the Laplacian on simply connected planar domains ($p=N=2$) discussed earlier in Section 1.3. Our techniques allow to generalize these results to arbitrary dimensions without however the explicit constants.

Also notice that the result can not hold if $p=N-1$, as noted by Hayman for the case $p=2, N=3$.

\begin{remark}
The proof of Theorem \ref{CurPosCap} also holds provided that the connected components of $\Omega^\complement$ are "large enough". It would be interesting to extend Theorem \ref{CurPosCap} for domains that are $k$-connected.
\end{remark}


\section{Proofs} \label{S5}

\subsection{Proof of Theorem \ref{ThmPrincipal}}\label{S5thmP}

In order to prove \eqref{EqPrincipal}, we need the following lemma:

\begin{lemma}\label{LemmePrincipal}
Let $p>N$, then for all $a \in \partial\Omega$, there exists $C_{N,p}>0$ such that
$$\displaystyle \int_{B_R(a)} |u_\lambda(x)|^p dx \leq C_{N,p} R^p \int_{B_R(a)} |\nabla u_\lambda(x)|^p dx$$
where $R>0$.
\end{lemma}

\begin{proof}

Since $p>N$, we know that every boundary point is regular (see for instance \cite[p. 19 and Section 6]{Lind3} for a discussion on regular boundary points). Hence, the boundary condition is attained in the classical sense. Therefore, $\forall a \in \partial\Omega$, $u_\lambda(a) = 0$. Pick any such $a$. Let $R$ be any positive number. We want to show that there exists $C_{N,p}>0$ such that
$$\displaystyle \int_{B_R(a)} |u_\lambda(x)|^p dx \leq C_{N,p} R^p \int_{B_R(a)} |\nabla u_\lambda(x)|^p dx.$$

Notice that $B_R(a) \cap \Omega^\complement $ may not be empty. In such case, we extend $u_\lambda$ by zero outside $\Omega$. Since $p>N$, by a weaker version of Sobolev-Morrey result (see \cite[p. 283]{Evans}), for every $x\in B_R(a)$, we have that
$$|u_\lambda(x)| = |u_\lambda(x) - u_\lambda(a)| \leq C_{N,p} ||\nabla u_\lambda||_{L^p(B_r(x) \cap B_r(a))} |x-a|^{1 - N/p},$$
where $r = |x - a| \leq R$. Therefore, since $x \in B_R(a)$, we have that $B_r(x) \cap B_r(a) \subset B_R(a)$, thus, we get that
$$ |u_\lambda(x)| \leq C_{N,p} R^{1-N/p} ||\nabla u_\lambda||_{L^p(B_R(a))} .$$
Raising to the power $p$ and integrating over $B_R(a)$ yield that
\begin{eqnarray}\label{consecmor}
\displaystyle \int_{B_R(a)} |u_\lambda(x)|^p dx & \leq & C_{N,p} Vol(B_R(a)) R^{p-N} \int_{B_R(a)} |\nabla u_\lambda(x)|^p dx \\
& \leq & C_{N,p} R^p \int_{B_R(a)} |\nabla u_\lambda(x)|^p dx. \nonumber
\end{eqnarray}
\end{proof}

Following Hayman's argument, the next step consists of using \cite[Lemma 5]{Hay},
\begin{lemma}\label{LemmeHaymann}
$\Omega$ can be covered by balls $B_r(x)$ for $x\in \partial \Omega, r = \rho_\Omega (1+\sqrt{N})$ such that the $B_r(x)$ can be divided into at most $C_2 \leq (2\sqrt{N}+4)^N$ subsets in such way that different balls in the same subset are disjoint.
\end{lemma}
Covering $\Omega$ by such balls combined to \eqref{LemmePrincipal}, one gets that
\begin{eqnarray*}
\int_{\Omega} u^p dV & \leq & C(N,p) r^p \int_\Omega |\nabla u|^p dV \\
& \leq & C_1(N,p)  \rho_\Omega^p \int_\Omega |\nabla u|^p dV,
\end{eqnarray*}
for all $u \in C_0^\infty(\Omega)$. This concludes the proof of Theorem \ref{ThmPrincipal}.


\subsection{Proof of Theorem \ref{CurPosCap}}\label{S5thmS}

Let $\delta = 4 \rho_{\Omega} (1 + \sqrt{N})$. For any $x \in \partial
\Omega$, let $B_\delta(x)$ denote a ball of radius $\delta$ centered at $x$.
Let $K$ denote the connected component of $\closure{\Omega^\complement} \cap
\closure{B}_\delta(x)$ containing $x$. Notice that
$\operatorname{diam}(K)\geq \delta$. Indeed, by contradiction, suppose that
$\operatorname{diam}(K) < \delta$. Then, $K \cap \partial B_\delta(x) =
\emptyset$, which is equivalent to saying that $\closure{\Omega^\complement}
\cap \partial \closure{B}_\delta(x)= \emptyset$. The latter implies that
$\closure{\Omega^\complement}$ is bounded, a contradiction since $\partial
\Omega$ is connected and $\Omega$ is bounded.

Let $r = \rho_{\Omega} (1 + \sqrt{N}) = \frac{\delta}{4}$. We need to use a lemma to obtain a lower bound on the capacity of $K$, namely \cite[Lemma 5.2]{BuTr}, which states that:
\begin{lemma}
Let $p>N-1$. Let $K \subset \mathbb{R}^N$ be a compact, connected set. Then, for all $x\in K$ and $r < \frac{1}{2} \operatorname{diam}(K)$, we have
\begin{equation*}
\operatorname{Cap}_p(K \cap \closure{B}_r(x) , B_{2r}(x)) \geq
\frac{\operatorname{Cap}_p([0,1] \times \{0\}^{N-1},
B_{2}(x))}{\operatorname{Cap}_p(\closure{B}_1(0),  B_2(0))}
\operatorname{Cap}_p(\closure{B}_r (x),  B_{2r}(x)).
\end{equation*}
\end{lemma}
Since $\operatorname{diam}(K) \geq \delta$, we have that $\frac{1}{2}
\operatorname{diam}(K) \geq \frac{1}{2} \delta > r$, which allows us to use
the latter lemma, yielding that
\begin{equation}\label{eqCap}
\operatorname{Cap}_p(K \cap \closure{B}_r(x) , B_{2r}(x)) \geq \underbrace{
\frac{\operatorname{Cap}_p([0,1] \times \{0\}^{N-1},
B_{2}(x))}{\operatorname{Cap}_p(\closure{B}_1(0),  B_2(0))} }_{C_{N,p}}
\operatorname{Cap}_p(\closure{B}_r (x),  B_{2r}(x)).
\end{equation}
Let $F = K \cap \closure{B}_r(x) = \closure{\Omega^\complement} \cap
\closure{B}_r(x)$. Note that $u \equiv 0$ on $F \subset \closure{B}_r(x)$.
Therefore, using \cite[Theorem 14.1.2]{Maz}, we get that
\begin{equation}\label{PoinEstim}
\int_{B_r(x)} |u|^p dV \leq \frac{K_{N,p} r^N}{\operatorname{Cap}_p(F,
B_{2r}(x))} \int_{B_r(x)} |\nabla u|^p dV,
\end{equation}
where $K_{N,p}$ is a non-explicit constant. Using \eqref{eqCap}, we get that
$$ \frac{1}{\operatorname{Cap}_p(F, B_{2r}(x))} \leq
\frac{1}{C_{N,p}\operatorname{Cap}_p(B_{r}(x), B_{2r}(x))} =
\frac{1}{C_{N,p}r^{N-p}\operatorname{Cap}_p(\closure{B}_1 (0),  B_2(0))}. $$
Combining the latter equation with \eqref{PoinEstim}, we get that
\begin{equation}\label{PoinEstim2}
\int_{B_r(x)} |u|^p dV \leq \frac{K_{N,p}
r^p}{C_{N,p}\operatorname{Cap}_p(\closure{B}_{1}(0),  B_{2}(0))} \int_{B_r(x)}
|\nabla u|^p dV.
\end{equation}
Following Hayman's argument, the next step consists of using Lemma
\ref{LemmeHaymann}. Co\-vering $\Omega$ by such balls combined to
\eqref{PoinEstim2}, one gets that
\begin{eqnarray*}
\int_{\Omega} u^p dV & \leq & K_{1,N,p} K_{2,N,p} r^p \int_\Omega |\nabla
u|^p dV \\
& \leq & K_{3,N,p}  \rho_\Omega^p \int_\Omega |\nabla u|^p dV,
\end{eqnarray*}
for all $u \in C_0^\infty(\Omega)$, yielding the desired result.

\proof[Acknowledgements]
This research is part of the author's Ph.D. thesis under the supervision of Iosif Polterovich. The author is grateful to Dorin Bucur for his insights on the $p$-capacity and for suggesting that Theo\-rems \ref{ThmPrincipal} and \ref{CurPosCap} should hold. Moreover, he wishes to thank Dan Mangoubi for useful discussions.

\end{document}